\documentclass[a4paper,12pt]{article}
\usepackage{amsfonts}
\usepackage{amstext}
\usepackage{amsthm}
\usepackage{amsmath}
\usepackage{amssymb}
\usepackage{latexsym}
\usepackage{graphicx}
\usepackage[latin1]{inputenc}
\newcommand{\R}{\Bbb{R}}

\newcommand{\s}{\Bbb{S}}

\newcommand{\bo}{\mathbb B}
\newtheorem{teor}{Theorem}[section]
\newtheorem{propo}{Proposition}[section]

\newtheorem{cor}{Corollary}[section]

\newcommand{\n}{\noindent}

\begin{document}

\title{Least energy solutions for affine $p$-Laplace equations involving subcritical and critical nonlinearities
\footnote{2020 Mathematics Subject Classification: 35A15; 35J20; 35J66}
\footnote{Key words: Affine variational problems, affine Sobolev inequalities, affine $p$-Laplace operator, affine Brezis-Nirenberg problem}
}

\author{\textbf{Edir Junior Ferreira Leite \footnote{\textit{E-mail addresses}:
edirjrleite@ufv.br (E.J.F. Leite)}}\\ {\small\it Departamento de Matem\'{a}tica,
Universidade Federal de Vi\c{c}osa,}\\ {\small\it CCE, 36570-900, Vi\c{c}osa, MG, Brazil}\\
\textbf{Marcos Montenegro \footnote{\textit{E-mail addresses}:
montene@mat.ufmg.br (M.
Montenegro)}}\\ {\small\it Departamento de Matem\'{a}tica,
Universidade Federal de Minas Gerais,}\\ {\small\it Caixa Postal
702, 30123-970, Belo Horizonte, MG, Brazil}}

\date{}{%{\it Preprint - December 12, 2009}}

\maketitle

\markboth{abstract}{abstract}
\addcontentsline{toc}{chapter}{abstract}

\hrule \vspace{0,2cm}

\n {\bf Abstract}
The paper is concerned with Lane-Emden and Brezis-Nirenberg problems involving the affine $p$-laplace nonlocal operator $\Delta_p^{\cal A}$, which has been introduced in \cite{HJM5} driven by the affine $L^p$ energy ${\cal E}_{p,\Omega}$ from convex geometry due to Lutwak, Yang and Zhang \cite{LYZ2}. We are particularly interested in the existence and nonexistence of positive $C^1$ solutions of least energy type. Part of the main difficulties are caused by the absence of convexity of ${\cal E}_{p,\Omega}$ and by the comparison ${\cal E}_{p,\Omega}(u) \leq \Vert u \Vert_{W^{1,p}_0(\Omega)}$ generally strict.

\vspace{0.5cm}
\hrule\vspace{0.2cm}

\section{Introduction and main results}

In two seminal works, Lutwak, Yang and Zhang established the famous sharp affine $L^p$ Sobolev inequality, namely, in \cite{Z} for $p = 1$ and in \cite{LYZ2} for $1 < p < n$. Precisely, for any $u \in {\cal D}^{1,p}(\R^n)$, it states that
\begin{equation} \label{ASob}
\Vert u \Vert_{L^{p^*}(\R^n)} \leq K_{n,p}\, {\cal E}_p(u),
\end{equation}
where $p^* = \frac{np}{n-p}$ is the Sobolev critical exponent, ${\cal D}^{1,p}(\R^n)$ denotes the usual space of weakly differentiable functions in $\R^n$ endowed with the $L^p$ gradient norm and ${\cal E}_p(u)$ stands for the affine $L^p$ energy expressed for any $p \geq 1$ as
\[
{\cal E}_p(u) = \alpha_{n,p} \left( \int_{\s^{n-1}} \left( \int_{\R^n} | \nabla_\xi u(x) |^p\, dx \right)^{-\frac{n}{p}}\, d\sigma(\xi)\right)^{-\frac 1n}
\]
with $\alpha_{n,p} = \left( 2 \omega_{n+p-2} \right)^{-1/p} \left(n \omega_n \omega_{p-1}\right)^{1/p} \left(n \omega_n\right)^{1/n}$. Here, $\nabla_\xi u(x)$ represents the directional derivative $\nabla u(x) \cdot \xi$ with respect to the direction $\xi \in \s^{n-1}$ and $\omega_k$ is the volume of the unit Euclidean ball in $\R^k$.

The value of the optimal constant in \eqref{ASob} is well known and given for $p = 1$ by
\[
K_{n,1} = \pi^{-\frac 12} n^{-1} \Gamma \left(\frac{n}{2}+1\right)^{\frac{1}{n}}
\]
and for $1 < p < n$ by
\[
K_{n,p} = \pi^{-\frac 12} n^{-\frac 1p} {\left(\frac{p-1}{n-p}\right)^{1-\frac 1p} \left(\frac{\Gamma \left(\frac{n}{2}+1\right) \Gamma (n)}{\Gamma \left(-\frac{n}{p}+n+1\right) \Gamma \left(\frac{n}{p}\right)}\right)^{\frac{1}{n}}}.
\]
Moreover, the corresponding extremal functions are precisely $u(x) = a \chi_{\bo} \left( b A(x - x_0) \right)$ for $p = 1$ (i.e. multiples of characteristics functions of ellipsoids) and
\[
u(x) = a\left(1 + b|A(x - x_0)|^\frac p{p-1}\right)^{1-\frac pn},
\]
for $a \in \R$, $b > 0$, $x_0 \in \R^n$ and $A \in SL(n)$, where $\bo$ is the unit Euclidean ball in $\R^n$ and $SL(n)$ denotes the special linear group of $n \times n$ matrices with determinant equal to $1$.

Ever since various improvements and new affine functional inequalities have emerged in a very comprehensive literature. For affine Sobolev type inequalities on the whole $\R^n$ we refer to \cite{HJM1, HJM3, LYZ2, Z}, for affine Sobolev type trace inequalities on the half space $\R^n_+$ to \cite{NHJM, N1} and for affine $L^p$ Poincaré-Sobolev inequalities on bounded domains to \cite{LM} for $p=1$ and to \cite{ST} for $p=2$. For other affine functional inequalities we mention the references \cite{CLYZ, NHJM, HS, HJM1, HJM3, HJM4, HJS, KS1, LM, LXZ, LYZ1, LYZ2, LYZ3, Lv, N, N1, N2, Wa, Z}.

More recently, Haddad, Jiménez and Montenegro \cite{HJM5} investigated affine $L^p$ Poincaré inequalities on bounded domains $\Omega$ and solved the associated Faber-Krahn problem for any $p \geq 1$. Its solution demanded the introduction of a differential operator $\Delta^{\cal A}_p$ for $p > 1$ closely related to the affine $L^p$ energy on $W_0^{1,p}(\Omega)$ given by
\[
{\cal E}_{p,\Omega}(u) = \alpha_{n,p} \left( \int_{\s^{n-1}} \left( \int_{\Omega} | \nabla_\xi u(x) |^p\, dx \right)^{-\frac{n}{p}}\, d\sigma(\xi)\right)^{-\frac 1n},
\]
where $W_0^{1,p}(\Omega)$ denotes the completion of the space $C^\infty_0(\Omega)$ of smooth functions compactly supported in $\Omega$ with respect to the norm
\[
\Vert u \Vert_{W^{1,p}_0(\Omega)} = \Vert \nabla u \Vert_{L^p(\Omega)} = \left( \int_\Omega \vert \nabla u \vert^p\, dx \right)^{\frac 1p}.
\]

The affine $p$-Laplace operator $\Delta^{\cal A}_p = \Delta^{\cal A}_{p, \Omega}$ on $W_0^{1,p}(\Omega) \setminus \{0\}$ is defined as the nonlocal quasilinear operator in divergence form:

\[
\Delta^{\cal A}_p u = -{\rm div} \left( H_{u}^{p-1}(\nabla u) \nabla H_{u}(\nabla u) \right),
\]
where

\[
H_{u}^p(\zeta) = \alpha_{n,p}^{-n}\; {\cal E}_{p, \Omega}^{n + p}(u) \int_{\s^{n-1}} \left( \int_{\Omega} | \nabla_\xi u(x) |^p\, dx \right)^{-\frac{n+p}{p}}  |\langle \xi, \zeta \rangle|^p\, d\sigma(\xi)\ \ {\rm for}\ \zeta \in \R^n.
\]
By using a key relation satisfied by ${\cal E}_{p,\Omega}$ which works specially in the case $p=2$, we point out that $\Delta^{\cal A}_2$ coincides with the affine Laplace operator introduced by Schindler and Tintarev in \cite{ST}. The name of $\Delta^{\cal A}_p$ is inspired on two fundamental properties presented in Section 4 of \cite{HJM5}. For $p > 1$, firstly it coincides with the standard $p$-Laplace operator $\Delta_p u = - {\rm div}(|\nabla u|^{p-2} \nabla u)$ for radial functions when $\Omega$ is a ball centered at the origin and, secondly, it verifies the affine invariance property $\Delta^{\cal A}_p (u \circ T) = (\Delta^{\cal A}_p u) \circ T$ on $T^{-1}(\Omega)$ for every $u \in W^{1,p}_0(\Omega)$ and $T \in SL(n)$.

Our main goal is the study of the boundary problem

\begin{equation} \label{1.1}
\left\{
\begin{array}{rlllr}
\Delta^{\cal A}_p u &=& u^{q - 1} + \lambda u^{p - 1} & {\rm in} & \Omega, \\
u&=&0 & {\rm on} & \partial \Omega,
\end{array}\right.
\end{equation}
where $\Omega$ is a bounded open subset of $\R^n$ with smooth boundary (e.g. of $C^{2,\alpha}$ class).

Throughout this work, it is assumed that $1 < p < n$, $p < q \leq p^*$ and $\lambda$ is a real parameter. Particularly important cases are the Lane-Emden problem ($\lambda = 0$) and the Brezis-Nirenberg problem ($q = p^*$) regarding the affine $p$-Laplace operator.

A nonnegative function $u_0 \in W^{1,p}_0(\Omega) \setminus \{0\}$ is said to be a weak solution of \eqref{1.1}, if
\begin{equation} \label{1.2}
\int_\Omega H_{u_0}^{p-1}(\nabla u_0) \nabla H_{u_0}(\nabla u_0) \cdot \nabla \varphi\; dx = \int_\Omega u_0^{q - 1} \varphi\; dx + \int_\Omega \lambda u_0^{p - 1} \varphi\; dx
\end{equation}
for all $\varphi \in W^{1,p}_0(\Omega)$. If $u_0 \in C^1(\overline{\Omega})$ and $u_0 > 0$ in $\Omega$, we simply say that $u_0$ is a positive $C^1$ solution.

The left-hand side of \eqref{1.2} is precisely equal to

\[
\frac{1}{p} \frac{d}{dt} {\cal E}^p_{p,\Omega}(u_0 + t \varphi) \mid_{t = 0},
\]
according to the computation of the directional derivative of ${\cal E}_{p,\Omega}^p(u)$ at $u_0 \in W^{1,p}_0(\Omega) \setminus \{0\}$ in the direction $\varphi \in W^{1,p}_0(\Omega)$ done in the proof of Theorem 10 of \cite{HJM5}.

We are particularly interested in the existence of positive $C^1$ solution and nonexistence of nontrivial nonnegative weak solution that attains the minimum level

\[
c_{\cal A} = \inf_{u \in W^{1,p}_0(\Omega) \setminus \{0\}}\frac{{\cal E}_{p,\Omega}^p (u) - \lambda \int_{\Omega} \vert u\vert^p\, dx}{\left( \int_\Omega \vert u\vert^q\, dx \right)^{\frac{p}{q}}}.
\]
Such a solution is called least energy type solution of \eqref{1.1}.

In \cite{HJM5}, the authors also showed that the operator $\Delta^{\cal A}_p$ on $W^{1,p}_0(\Omega)$ possesses a positive principal eigenvalue $\lambda^{\cal A}_{1,p} = \lambda^{\cal A}_{1,p}(\Omega)$ characterized variationally by
\[
\lambda^{\cal A}_{1,p} = \inf \{ {\cal E}^p_{p,\Omega}(u):\ u \in W^{1,p}_0(\Omega),\ \Vert u \Vert_{L^p(\Omega)} = 1 \}.
\]

Our main theorems are affine counterparts of pioneering results dealing with problems related to the $p$-Laplace operator.

\begin{teor} \label{T1}
If $p < q < p^*$, then the affine problem \eqref{1.1} admits a positive $C^1$ least energy solution for any $\lambda < \lambda^{\cal A}_{1,p}$.
\end{teor}

An interesting consequence of this result occurs when $\lambda = 0$. In effect, for any $q \in [1, p^*]$, the affine Sobolev inequality \eqref{ASob} yields the affine $L^q$ Poincaré-Sobolev inequality on $W^{1,p}_0(\Omega)$:

\begin{equation} \label{APSob}
\mu_{p,q}^{\cal A} \Vert u \Vert_{L^q(\Omega)} \leq {\cal E}_{p,\Omega}(u)
\end{equation}
for an optimal constant $\mu_{p,q}^{\cal A} = \mu_{p,q}^{\cal A}(\Omega)$. A nonzero function $u_0 \in W^{1,p}_0(\Omega)$ is said to be an extremal for \eqref{APSob}, if this inequality becomes equality. The existence of extremals has been established in \cite{HJM5} for $q = p$ and in \cite{HX} for $1 \leq q < p$. Using these results and Theorem \ref{T1}, we readily derive

\begin{cor} \label{C1}
The sharp affine inequality \eqref{APSob} admits extremal functions if, and only if, $1 \leq q < p^*$.
\end{cor}

\n Clearly, the cases $p < q < p^*$ follow from the equality $\mu_{p,q}^{\cal A} = c_{\cal A}$ and Theorem \ref{T1}. Already when $q = p^*$, exist no extremal as can easily be seen through usual arguments. In effect, a canonical rescaling immediately yields $\mu_{p,p^*}^{\cal A} = K_{n,p}^{-1}$. On the other hand, as described in the introduction, the extremal functions for the sharp affine Sobolev inequality \eqref{ASob} doesn't belong to $W^{1,p}_0(\Omega)$.

\begin{teor} \label{T2}
If $q = p^*$ and $n \geq p^2$, then the affine problem \eqref{1.1} admits a positive $C^1$ least energy solution for any $0 < \lambda < \lambda^{\cal A}_{1,p}$.
\end{teor}

\begin{teor} \label{T3}
If $q = p^*$ and $n < p^2$, then there exists a constant $\lambda_* > 0$ such that the affine problem \eqref{1.1} admits a positive $C^1$ least energy solution for any $\lambda_* < \lambda < \lambda^{\cal A}_{1,p}$.
\end{teor}

\begin{teor} \label{T4}
If $p < q \leq p^*$, then the affine problem \eqref{1.1} admits no nontrivial least energy weak solution for any $\lambda \geq \lambda^{\cal A}_{1,p}$.
\end{teor}

\begin{teor} \label{T5}
If $q = p^*$ and $\Omega$ is star-shaped, then the affine problem \eqref{1.1} admits no nontrivial nonnegative weak solution for any $\lambda \leq 0$.
\end{teor}

The well-known versions of all theorems from \ref{T1} to \ref{T5} for the $p$-Laplace operator were established at various times in history. For works associated to Theorem \ref{T1}, see \cite{AP1, GP, O}, to Theorems \ref{T2} and \ref{T3}, see \cite{BN} for $p = 2$ and \cite{AP1, AP2, GV} for $1 < p < n$ and to Theorem \ref{T5}, see \cite{P} for $p = 2$ and \cite{DMS, GV, PS1} for $1 < p < n$. Lastly, closely related to Theorem \ref{T4}, we recall the results of nonexistence of positive $C^1$ solution for any $\lambda \geq \lambda_{1,p} = \lambda_{1,p}(\Delta_p)$ and of existence of nontrivial $C^1$ solution for any $\lambda > \lambda_{1,p}$ different from all min-max type eigenvalues of $\Delta_p$ on $W^{1,p}_0(\Omega)$, see \cite{GV} and \cite{DL}, respectively. Whether these results are true or not when considering the operator $\Delta^{\cal A}_p$ is an open question. For other related references, we refer to the works \cite{AR, CV, CV1, E1, E, GY, HHZ}.

The proof of existence consists in finding minimizers for the quotient defining the minimum level $c_{\cal A}$. The main obstacles are caused by the affine term  ${\cal E}^p_{p,\Omega}$. More precisely, the functional $u \in W^{1,p}_0(\Omega) \mapsto {\cal E}^p_{p,\Omega}(u)$ is not convex and its geometry is non-coercive since there are unbounded sequences in $W^{1,p}_0(\Omega)$ with bounded affine $L^p$ energy. Examples of such sequences are constructed on the pages 17 and 18 of \cite{HJM5}. Already, the non-convexity follows from the reverse inequality of Proposition \ref{P5} which becomes strict for many functions in $W^{1,p}_0(\Omega)$. Despite the absence of an adequate variational structure, we prove in Theorem \ref{T6} that the referred energy functional is weakly lower semicontinuous on $W^{1,p}_0(\Omega)$.

The affine context also affects strongly the study of existence of nonnegative weak solution in the subcritical and critical cases. In both ones, we make use of an affine Rellich-Kondrachov compactness theorem, established recently by Tintarev \cite{T1} for $1 < p < n$ and by the authors \cite{LM} for $p = 1$, which states that the affine ball $B^{\mathcal A}_p(\Omega) = \{u \in W^{1,p}_0(\Omega):\ {\cal E}_{p,\Omega}(u) \leq 1\}$ is compact in $L^q(\Omega)$ for every $1 \leq q < p^*$ (Theorem \ref{T7}), and also of a quite useful consequence of its proof (Corollary \ref{C2}). When $q = p^*$, it is well-known that the claim of compactness usually fails. However, using additional tools, we prove that minimizing sequences of $c_{\cal A}$ are compact in $L^{p^*}(\Omega)$ for lower energy levels $c_{\cal A}$, see Propositions \ref{P5}, \ref{P6} and \ref{P7}.

In Section 3, we focus on the $C^1$ regularity and positivity of weak solutions of \eqref{1.1} as well as a related Pohozaev type identity to be used in proof of Theorem \ref{T5}.

It is worth mentioning that (iii) of Proposition \ref{P1} in the next section will play a fundamental role in the proof of the most ingredients quoted above.

\section{Brief summary on the variational setting}
For $1 < p < n$ and $p < q \leq p^*$, let $\Phi_{\cal A} : W^{1,p}_0(\Omega) \rightarrow \R$ be the functional given by
\[
\Phi_{\cal A}(u) = {\cal E}_{p,\Omega}^p (u) - \lambda \int_{\Omega} \vert u\vert^p\; dx
\]
and $c_{\cal A} = \inf\limits_{u \in X} \Phi_{\cal A}(u)$ be its least energy level on the set $X = \{u \in W^{1,p}_0(\Omega):\ \Vert u \Vert_{L^q(\Omega)} = 1\}$.

Clearly, $\Phi_{\cal A}$ is well-defined and $c_{\cal A}$ is always finite for any $\lambda \in \R$, since the embedding $W^{1,p}_0(\Omega) \hookrightarrow L^{p^*}(\Omega)$ is continuous and the inequality for $p \geq 1$
\begin{equation} \label{comp}
{\cal E}_{p, \Omega} (u) \leq \Vert \nabla u \Vert_{L^p(\Omega)}
\end{equation}
holds for any $u \in W^{1,p}_0(\Omega)$, see page 14 of \cite{HJM5}.

Let $u_0 \in X$ be a nonnegative minimizer of $\Phi_{\cal A}$ and assume that $c_{\cal A} > 0$. Using the directional derivative of the affine term ${\cal E}_{p,\Omega}^p (u)$ described in the introduction, one easily checks that $u_0$ is a nonnegative weak solution of the problem

\begin{equation} \label{EL}
\left\{
\begin{array}{rlllr}
\Delta^{\cal A}_p u &=& c_{\cal A} u^{q - 1} + \lambda u^{p - 1} & {\rm in} & \Omega, \\
u&=&0 & {\rm on} & \partial \Omega.
\end{array}\right.
\end{equation}
Consequently, thanks to $(p-1)$-homogeneity of $\Delta^{\cal A}_p$, a straightforward argument implies that $c_{\cal A}^{\frac{1}{q-p}} u_0$ is a nonnegative weak solution of \eqref{1.1} of least energy type.

The first key point in the study of existence of minimizers for $c_{\cal A}$ is the weak lower semicontinuity of the functional $u \in W^{1,p}_0(\Omega) \mapsto {\cal E}_{p,\Omega}(u)$. The property has been recently proved in \cite{HJM5} through an elegant argument based on its Theorem 9 and Lemma 1. We next provide an alternative elementary proof.

\begin{teor} \label{T6}
If $u_k \rightharpoonup u_0$ weakly in $W^{1,p}_0(\Omega)$, then
\[
{\cal E}_{p,\Omega} (u_0) \leq \liminf_{k \rightarrow \infty} {\cal E}_{p,\Omega} (u_k).
\]
\end{teor}

We make use of the following result which will also play a strategic role along the work:

\begin{propo} \label{P1}
Let $u \in W^{1,p}_0(\Omega)$ with $p \geq 1$. The sentences are equivalent:

\begin{itemize}
\item[{\rm (i)}] $u = 0$;
\item[{\rm (ii)}] ${\cal E}_{p,\Omega} (u) = 0$;
\item[{\rm (iii)}] $\Psi_\xi(u) = 0$ for some $\xi \in \s^{n-1}$, where $\Psi_\xi(u) = \int_{\Omega} | \nabla_\xi u(x) |^p\, dx$.
\end{itemize}
\end{propo}

\begin{proof}
Clearly, by \eqref{comp}, (i) implies (ii). If the claim (iii) occurs, then $\nabla_\xi u(x) = 0$ almost everywhere in $\Omega$ for some $\xi \in \s^{n-1}$, so $u$ is constant on line segments in $\Omega$ in the direction $\xi$. Then, since $u$ has zero trace on $\partial \Omega$, it follows the claim (i).

It remains to show that (ii) implies (iii). In fact, arguing by contradiction, assume that $\Psi_{\xi}(u) > 0$ for all $\xi \in \s^{n-1}$. Thanks to the continuity of the map $\xi \in \s^{n-1} \mapsto \Psi_{\xi}(u)$, there exists a constant $c > 0$ so that $\Psi_{\xi}(u) \geq c$ for all $\xi \in \s^{n-1}$. But the lower bound immediately yields ${\cal E}_{p,\Omega}(u) \geq c^{1/p} \alpha_{n,p} (n \omega_n)^{-1/n} > 0$. Thus, (ii) fails and we end the proof.
\end{proof}

\begin{proof}[Proof of Theorem \ref{T6}]
Let $u_k$ be a sequence converging weakly to $u_0$ in $W^{1,p}_0(\Omega)$. If $u_0 = 0$ then, by (ii) of Proposition \ref{P1}, the statement follows trivially.

Assume $u_0 \neq 0$. Thanks to the convexity of the functional $u \in W^{1,p}_0(\Omega) \mapsto \Psi_{\xi}(u)$, for any $\xi \in \s^{n-1}$, we have

\begin{equation} \label{linf}
\Psi_{\xi}(u_0) \leq \liminf_{k \rightarrow \infty} \Psi_{\xi}(u_k).
\end{equation}
We now ensure the existence of a constant $c_0 > 0$ and an integer $k_0 \geq 1$, both independent of $\xi \in \s^{n-1}$, such that

\begin{equation} \label{loweri}
\Psi_\xi(u_k) \geq c_0
\end{equation}
for all $k \geq k_0$. Otherwise, module a renaming of indexes, we get a sequence $\xi_k \in \s^{n-1}$ such that $\xi_k \rightarrow \tilde{\xi}$ and $\Psi_{\xi_k}(u_k) \leq k^{-1}$. Since $u_k$ is bounded in $W^{1,p}_0(\Omega)$, we find a constant $C_1 > 0$ such that
\[
\Psi_{\tilde{\xi}}(u_k) \leq C_1 \Vert \xi_k - \tilde{\xi} \Vert^p + 2^{p-1}k^{-1}.
\]
Letting $k \rightarrow \infty$ and using \eqref{linf}, we get $\Psi_{\tilde{\xi}}(u_0) = 0$. But, by Proposition \ref{P1}, we obtain the contradiction $u_0 = 0$, and so \eqref{loweri} is satisfied.

Finally, combining \eqref{linf}, \eqref{loweri} and Fatou's lemma, we derive

\[
\int_{\s^{n-1}} \Psi_{\xi}(u_0)^{-\frac{n}{p}}\, d\sigma(\xi) \geq \int_{\s^{n-1}} \limsup_{k \rightarrow \infty} \Psi_{\xi}(u_k)^{-\frac{n}{p}}\, d\sigma(\xi) \geq \limsup_{k \rightarrow \infty} \int_{\s^{n-1}} \Psi_{\xi}(u_k)^{-\frac{n}{p}}\, d\sigma(\xi),
\]
and hence

\begin{eqnarray*}
{\cal E}_{p,\Omega} (u_0) &=& \alpha_{n,p} \left( \int_{\s^{n-1}} \Psi_{\xi}(u_0)^{-\frac{n}{p}}\, d\sigma(\xi) \right)^{-\frac{1}{n}} \\
&\leq& \liminf_{k \rightarrow \infty} \alpha_{n,p} \left( \int_{\s^{n-1}} \Psi_{\xi}(u_k)^{-\frac{n}{p}}\, d\sigma(\xi) \right)^{-\frac{1}{n}} \\
&=& \liminf_{k \rightarrow \infty} {\cal E}_{p,\Omega} (u_k).
\end{eqnarray*}
\end{proof}

The second point was recently established by Tintarev for $1 < p < n$ (see Theorem 6.5.3 of \cite{T1}) and by the authors for $p = 1$ (see Theorem 4.1 of \cite{LM}) and is stated as follows.

\begin{teor} \label{T7}
Let $1 \leq p < n$ and $B^{\mathcal A}_p(\Omega) = \{u \in W^{1,p}_0(\Omega):\ {\cal E}_{p,\Omega}(u) \leq 1\}$. The set $B^{\cal A}_p(\Omega)$ is compact in $L^q(\Omega)$ for every $1 \leq q < p^*$.
\end{teor}

Given a sequence $u_k$ in $B^{\cal A}_p(\Omega)$, the proof of this result involves the existence of matrices $T_k \in SL(n)$ such that $u_k \circ T_k$ is bounded in ${\cal D}^{1,p}(\R^n)$. If $T_k \rightarrow \infty$ is proved there that $u_k \rightarrow 0$ in $L^q(\Omega)$. This fact leads us to a simple consequence that deserves to be highlighted.

\begin{cor} \label{C2}
Let $u_k$ be a sequence in $B^{\cal A}_p(\Omega)$ such that $u_k \rightarrow u_0$ strongly in $L^q(\Omega)$ for some $1 \leq q < p^*$. If $u_0 \neq 0$, then $u_k$ is bounded in $W^{1,p}_0(\Omega)$.
\end{cor}

\section{Some properties of weak solutions}

We next present some important properties satisfied by weak solutions of \eqref{1.1}. We begin with a result on $C^1$ regularity regarding critical nonlinearities.

\begin{propo} \label{P2}
Let $\Omega$ be a bounded domain with $C^{2,\alpha}$ boundary and $f: \Omega \times \R \rightarrow \R$ be a $C^1$ function satisfying
\[
\vert f(x,t) \vert \leq b(x) (\vert t \vert^{p^*-1} + 1)
\]
for all $(x,t) \in \Omega \times \R$, where $b \in L^{\infty}(\Omega)$. If $u_0 \in W^{1,p}_0(\Omega) \setminus \{0\}$ is a weak solution of the problem
\[
\left\{
\begin{array}{rlllr}
\Delta^{\cal A}_p u &=& f(x,u) & {\rm in} & \Omega, \\
u&=&0 & {\rm on} & \partial \Omega,
\end{array}\right.
\]
then $u_0 \in C^{1,\alpha}(\overline{\Omega})$.
\end{propo}

\begin{proof}
We first show that $u_0 \in L^s(\Omega)$ for any $s \geq 1$. Note that $u_0$ can be seen as a weak solution of

\[
\left\{
\begin{array}{rlllr}
\Delta^{\cal A}_p u &=& g(x,u) & {\rm in} & \Omega, \\
u&=&0 & {\rm on} & \partial \Omega,
\end{array}\right.
\]
where $g: \Omega \times \R \rightarrow \R$ is given by

\[
g(x,t) = \frac{f(x, u_0(x))}{\vert u_0(x) \vert^{p-1} + 1} (\vert t \vert^{p-1} + 1).
\]
Clearly, $g$ is a Caratheodory function and satisfies $\vert g(x,t)\vert \leq b_0(x) (\vert t\vert^{p-1} + 1)$ for all $(x,t) \in \Omega \times \R$, where $b_0$ belongs to $L^{n/p}(\Omega)$, since $\vert f(x,u_0) \vert \leq b(x) (\vert u_0 \vert^{p^*-p} \vert u_0 \vert^{p-1} + 1)$ and $\vert u_0\vert^{p^*-p} \in L^{n/p}(\Omega)$.

Following now the proof of Proposition 1.2 of \cite{GV}, we consider the test function $\varphi = \psi_k(u_0)$ for each positive number $k$, where $\psi_k(t) = \int_0^t \eta^\prime_k(\theta)^p\, d\theta$ and $\eta_k$ is the $C^1$ function given by $\eta_k(\theta) = {\rm sign}(\theta) \vert \theta \vert^{\frac{s}{p}}$ for $\vert \theta \vert \leq k$ and by $\eta_k(\theta) = {\rm sign}(\theta) [\frac{s}{p} k^{\frac{s}{p} - 1} \vert \theta \vert + (1 - \frac{s}{p}) k^{\frac{s}{p}}]$ for $\vert \theta \vert > k$. Then, we derive
\[
\int_\Omega \eta_k^\prime(u_0)^p\; H_{u_0}^{p-1}(\nabla u_0) \nabla H_{u_0}(\nabla u_0) \cdot \nabla u_0\, dx = \int_\Omega g(x,u_0) \psi_k(u_0)\, dx.
\]
On the other hand, using the $1$-homogeneity of $H_{u_0}(\zeta)$ on $\zeta$ and (iii) of Proposition \ref{P1} ($u_0 \neq 0$), we obtain a constant $c_0 > 0$, depending on $u_0$, such that
\[
\int_\Omega \eta_k^\prime(u_0)^p\, H_{u_0}^{p-1}(\nabla u_0) \nabla H_{u_0}(\nabla u_0) \cdot \nabla u_0\, dx = \int_\Omega \eta_k^\prime(u_0)^p\; H_{u_0}^p(\nabla u_0)\; dx \geq c_0 \int_\Omega \eta_k^\prime(u_0)^p\; \vert \nabla u_0 \vert^p\, dx.
\]
Thanks to this inequality, the same arguments as in \cite{GV}, consisting in the well-known De Giorgi-Nash-Moser's iterative scheme, can be applied and the first claim follows. Consequently, $f(x,u_0) \in L^s(\Omega)$ for any $s > n/p$ and thus Proposition 4 of \cite{HJM5} gives $u_0 \in C^{1, \alpha}(\overline{\Omega})$.
\end{proof}

The next result concerns with strong maximum principle and Hopf's lemma for an equation associated to the operator $\Delta^{\cal A}_p$.

\begin{propo} \label{P3}
Let $\Omega$ be a bounded domain with $C^{2,\alpha}$ boundary and $u_0 \in W^{1,p}_0(\Omega) \setminus \{0\}$ be a weak solution of $\Delta^{\cal A}_p u = \lambda |u|^{p-2} u + f(x)$ in $\Omega$, where $\lambda \leq 0$ and $f \in L^\infty(\Omega)$. If $f \geq 0$ in $\Omega$ and $f \not\equiv 0$ in $\Omega$, then $u_0 > 0$ in $\Omega$ and $\frac{\partial u_0}{\partial \nu} < 0$ on $\partial \Omega$, where $\nu$ denotes the outward unit norm field to $\Omega$.
\end{propo}

\begin{proof}
Note that the assumption $f \not\equiv 0$ in $\Omega$ implies $u_0 \not\equiv 0$ in $\Omega$. We first assert that $u_0 \geq 0$ in $\Omega$. Taking the test function $u_0^- = \min\{u_0, 0\} \in W^{1,p}_0(\Omega)$ in the equation and using the assumptions of proposition and the relation $H_{u_0}^{p-1}(\zeta) \nabla H_{u_0}(\zeta) \cdot \zeta = H^p_{u_0}(\zeta)$, we get

\[
\int_\Omega H^p_{u_0}(\nabla u_0^-)\, dx \leq 0,
\]
and since $u_0 \not\equiv 0$, by Proposition \ref{P1}, $u_0^- = 0$ in $\Omega$, in other words, $u_0$ is nonnegative.

Consider now the operator ${\mathcal L}_{p, u_0} u := -{\rm div} \, {\cal H}(\nabla u)$ on $W_0^{1,p}(\Omega)$, where ${\cal H}(\zeta) = H_{u_0}^{p-1}(\zeta) \nabla H_{u_0}(\zeta)$ for $\zeta \in \R^n$. Since $u_0$ is nonzero, by Proposition \ref{P1} and Cauchy-Schwartz inequality, there are constants $c_0, C_0 > 0$ such that $c_0 \leq \Vert \nabla_\xi u_0 \Vert_{L^p(\Omega)} \leq C_0$ for every $\xi \in \s^{n-1}$. Thanks to these inequalities, the operator ${\mathcal L}_{p, u_0}$ is uniformly elliptic, that is, there exist constants $c_1, C_1 > 0$ such that

\begin{itemize}
\item[(a)] $\sum^n_{i,j=1} \frac{\partial {\cal H}_i(\zeta)}{\partial \zeta_j} \eta_i \eta_j \geq c_1 |\zeta|^{p-2} |\eta|^2$;
\item[(b)] $\sum^n_{i,j=1} \left| \frac{\partial {\cal H}_i(\zeta)}{\partial \zeta_j} \right| \leq C_1 |\zeta|^{p-2}$
\end{itemize}
for every $\zeta \in \R^n \setminus \{0\}$ and $\eta \in \R^n$, where ${\cal H}_i(\zeta)$ is the $i$-th component of ${\cal H}(\zeta)$. For the computations of (a) and (b), see pages 25 and 26 of \cite{HJM5}.

Since $f \in L^\infty(\Omega)$, by Proposition \ref{P2}, we know that $u_0 \in C^1(\overline{\Omega})$. Then, since $u_0 \geq 0$ in $\Omega$ and $u_0 \not\equiv 0$ in $\Omega$, evoking the strong maximum principle (Proposition 3.2.2 of \cite{To}) and Hopf's lemma (Proposition 3.2.1 of \cite{To}) for $C^1$ super-solutions of quasilinear elliptic equations for operators satisfying (a) and (b) (see also \cite{PS}), we derive the two desired statements.
\end{proof}

Finally, we shall need a Pohozaev type identity satisfied by weak solutions of \eqref{1.1}.

\begin{propo} \label{P4}
Let $\Omega$ be a bounded domain with $C^{2, \alpha}$ boundary, $f: \R \rightarrow \R$ be a continuous function such that $\vert f(t) \vert \leq b(\vert t \vert^{p^*-1} + 1)$ for all $t \in \R$, where $b > 0$ is a constant, and $F(t) = \int_0^t f(\theta)\, d\theta$. Then, the integral identity

\[
\left( \frac{1}{p} - 1 \right) \int_{\partial \Omega} H_{u_0}^p(\nabla u_0) (x \cdot \nu)\, d\sigma = \left( \frac{n}{p} - 1 \right) \int_\Omega u_0 f(u_0)\, dx - n \int_\Omega F(u_0)\, dx
\]
holds for any nontrivial weak solution $u_0 \in W^{1,p}_0(\Omega)$ of

\begin{equation} \label{1.3}
\left\{
\begin{array}{rlllr}
\Delta^{\cal A}_p u &=& f(u) & {\rm in} & \Omega, \\
u&=&0 & {\rm on} & \partial \Omega.
\end{array}\right.
\end{equation}
\end{propo}

\begin{proof}
Let $u_0 \in W^{1,p}_0(\Omega)$ be a nontrivial weak solution of \eqref{1.3}. By Proposition \ref{P2}, we know that $u_0 \in C^1(\overline{\Omega})$. Since $\zeta \in \R^n \mapsto H^p_{u_0}(\zeta)$ is strictly convex, we can apply the Pohozaev identity established by Degiovanni, Musesti and Squassina for $C^1$ solutions of

\begin{equation} \label{1.4}
\left\{
\begin{array}{rlllr}
-{\rm div} \left( {\cal H}(\nabla u) \right) &=& f(u) & {\rm in} & \Omega, \\
u&=&0 & {\rm on} & \partial \Omega,
\end{array}\right.
\end{equation}
where ${\cal H}(\zeta) = H_{u_0}^{p-1}(\zeta) \nabla H_{u_0}(\zeta)$. Precisely, evoking Theorems 1 and 2 of \cite{DMS} with the choice $h(x) = x$ and $a(x) = a$, where $a$ is an arbitrary constant, we have

\begin{eqnarray*}
&& \int_{\partial \Omega} \left( \frac{1}{p} H_{u_0}^p(\nabla u_0) -  H_{u_0}^{p-1}(\nabla u_0) \nabla H_{u_0}(\nabla u_0) \cdot \nabla u_0 \right)(x \cdot \nu)\, d\sigma\\
&& = \frac{n}{p} \int_\Omega H_{u_0}^p(\nabla u_0)\, dx - \int_\Omega H_{u_0}^{p-1}(\nabla u_0) \nabla H_{u_0}(\nabla u_0) \cdot \nabla u_0\, dx\\
&&\ \ \ - a \int_\Omega H_{u_0}^{p-1}(\nabla u_0) \nabla H_{u_0}(\nabla u_0) \cdot \nabla u_0\, dx + \int_\Omega \left( x \cdot \nabla u_0 + a u_0 \right) f(u_0)\, dx.
\end{eqnarray*}
Using the relation $H_{u_0}^{p-1}(\zeta) \nabla H_{u_0}(\zeta) \cdot \zeta = H^p_{u_0}(\zeta)$ (which comes from the $1$-homogeneity of $H_{u_0}(\zeta)$) and that $u_0$ solves \eqref{1.4}, the above identity can be placed into the simpler form

\begin{eqnarray*}
\left( \frac{1}{p} - 1 \right) \int_{\partial \Omega} H_{u_0}^p(\nabla u_0) (x \cdot \nu)\, d\sigma &=& \left( \frac{n}{p} - 1 - a \right) \int_\Omega H_{u_0}^p(\nabla u_0)\, dx + \int_\Omega \left( x \cdot \nabla u_0 + a u_0 \right) f(u_0)\, dx\\
&=& \left( \frac{n}{p} - 1 \right) \int_\Omega u_0 f(u_0)\, dx + \int_\Omega \left( x \cdot \nabla u_0 \right) f(u_0)\, dx.
\end{eqnarray*}
On the other hand, the divergence theorem and the condition $u_0 = 0$ on $\partial \Omega$ lead to

\begin{eqnarray*}
\int_\Omega \left( x \cdot \nabla u_0 \right) f(u_0)\, dx &=& \frac{1}{2} \int_\Omega \nabla( \vert x \vert^2) \cdot \nabla F(u_0)\, dx \\
&=& - \frac{1}{2} \int_\Omega \Delta( \vert x \vert^2) F(u_0)\, dx + \int_{\partial \Omega} F(u_0) (x \cdot \nu)\, d\sigma \\
&=& -n \int_\Omega F(u_0)\, dx.
\end{eqnarray*}
Replacing this equality in the previous one, one gets the wished identity for weak solutions of \eqref{1.3}.
\end{proof}

\section{Proof of existence theorems}

We prove Theorems \ref{T1}, \ref{T2} and \ref{T3} by applying the direct method to the functional $\Phi_{\cal A}$ constrained to the set $X$. According to Section 2, it suffices to show that the least energy level $c_{\cal A} = \inf\limits_{u \in X} \Phi_{\cal A}(u)$ is positive and is achieved for some positive $C^1$ function in $X$.

We begin with the subcritical case, whose main ingredients are Theorems \ref{T6} and \ref{T7}, Corollary \ref{C2} and Propositions \ref{P2} and \ref{P3}.

\begin{proof}[Proof of Theorem \ref{T1}]
Let $u_k$ be a minimizing sequence of $\Phi_{\cal A}$ in $X$. Since $q > p$, by Hölder's inequality, $u_k$ is bounded in $L^p(\Omega)$, so the affine energy ${\cal E}_{p, \Omega} (u_k)$ is bounded too. Then, by Theorem \ref{T7}, there exists $u_0 \in W^{1,p}_0(\Omega)$ such that $u_k \rightarrow u_0$ strongly in $L^p(\Omega)$ and in $L^q(\Omega)$ once $q < p^*$. In particular, $u_0 \in X$ and thus, by Corollary \ref{C2}, $u_k$ is bounded in $W^{1,p}_0(\Omega)$. Passing to a subsequence, if necessary, one may assume that $u_k \rightharpoonup u_0$ weakly in $W^{1,p}_0(\Omega)$. Then, by Theorem \ref{T6}, we derive
\[
\Phi_{\cal A}(u_0) \leq \liminf_{k \rightarrow \infty} \Phi_{\cal A}(u_k) = c_{{\cal A}},
\]
and thus $u_0$ minimizes $\Phi_{\cal A}$ in $X$. Moreover, we can assume that $u_0$ is nonnegative, since $\vert u_0 \vert \in X$ and $\Phi_{\cal A}(\vert u_0 \vert) = \Phi_{\cal A}(u_0)$.

Using now the assumption $\lambda < \lambda^{\cal A}_{1,p}$ and the sharp affine $L^p$ Poincaré inequality (Theorem 4 of \cite{HJM5}), we get
\[
c_{\cal A} = \Phi_{\cal A}(u_0) = {\cal E}^p_{p,\Omega}(u_0) - \lambda \int_\Omega \vert u_0 \vert^p\, dx \geq \left( \lambda^{\cal A}_{1,p} - \lambda \right) \int_\Omega \vert u_0 \vert^p\, dx > 0.
\]
Therefore, $u_0$ is a nontrivial nonnegative weak solution of \eqref{EL}. Finally, by Propositions \ref{P2} and \ref{P3}, $u_0$ is a positive $C^1$ solution of \eqref{EL}.
\end{proof}

The existence of positive $C^1$ solutions to critical problems requires three more results. The first one considers the truncation for $h > 0$:

\[
T_h(s) = \min(\max(s,-h), h)\ \text{ and }\ R_h(s) = s - T_h(s).
\]
A simple computation gives $\Vert \nabla u \Vert^p_{L^p(\Omega)} = \Vert \nabla T_h u \Vert^p_{L^p(\Omega)} + \Vert \nabla R_h u \Vert^p_{L^p(\Omega)}$ for every $u \in W^{1,p}_0(\Omega)$. Unfortunately, the equality is not valid within the affine setting, but it is still possible to guarantee an inequality thanks to (iii) of Proposition \ref{P1}.

\begin{propo}\label{P5}
For any $u \in W^{1,p}_0(\Omega)$, we have
\[
{\cal E}^p_{p, \Omega}(u) \geq {\cal E}^p_{p, \Omega}(T_h u) + {\cal E}^p_{p, \Omega}(R_h u).
\]
\end{propo}

\begin{proof} From the definition of $T_h(s)$, we have $T_h u, R_h u \in W_0^{1,p}(\Omega)$ and $\Psi_\xi(u) = \Psi_\xi(T_h u) + \Psi_\xi(R_h u)$ for all $\xi \in \s^{n-1}$. Note that the decomposition implies ${\cal E}_{p, \Omega}(u) \geq {\cal E}_{p, \Omega}(T_h u)$ and ${\cal E}_{p, \Omega}(u) \geq {\cal E}_{p, \Omega}(R_h u)$. So, the statement follows in the case that ${\cal E}_{p, \Omega}(T_h u) = 0$ or ${\cal E}_{p, \Omega}(R_h u) = 0$.

Assume now that ${\cal E}_{p, \Omega}(T_h u)$ and ${\cal E}_{p, \Omega}(R_h u)$ are nonzero. By (iii) of Proposition \ref{P1}, we have $\Psi_\xi(T_h u)$, $\Psi_\xi(R_h u) > 0$ for all $\xi \in \s^{n-1}$. Then, by the reverse Minkowski inequality for negative exponents, we get
\begin{eqnarray*}
{\cal E}^p_{p, \Omega}(u) &=& \alpha_{n,p}^p \left( \int_{\s^{n-1}} \left( \Psi_\xi(T_h u) + \Psi_\xi(R_h u) \right)^{-\frac np} d\sigma(\xi)\right)^{-\frac pn} \\
&\geq& \alpha_{n,p}^p \left( \int_{\s^{n-1}} \left( \Psi_\xi(T_h u) \right)^{-\frac np}\, d\sigma(\xi)\right)^{-\frac pn} + \alpha_{n,p}^p \left( \int_{\s^{n-1}} \left( \Psi_\xi(R_h u) \right)^{-\frac np}\, d\sigma(\xi)\right)^{-\frac pn} \\
&=& {\cal E}^p_{p, \Omega}(T_h u) + {\cal E}^p_{p, \Omega}(R_h u)
\end{eqnarray*}
for every $u \in W^{1,p}_0(\Omega)$.
\end{proof}

\begin{propo}\label{P6}
Let $q = p^*$ and assume that $0 < c_{\cal A} < K_{n,p}^{-p}$, where $K_{n,p}$ is the best constant for the sharp affine $L^p$ Sobolev inequality \eqref{ASob}. Then, $\Phi_{\cal A}$ admits a minimizer $u_0$ in $X$.
\end{propo}

\begin{proof}
Let $u_k$ be a minimizing sequence of $\Phi_{\cal A}$ in $X$. Proceeding as in the proof of Theorem \ref{T1}, by Theorem \ref{T7}, $u_k \rightarrow u_0$ strongly in $L^p(\Omega)$, module a subsequence. One may also assume that $u_k \rightarrow u_0$ almost everywhere in $\Omega$ and $T_h u_k \rightharpoonup T_h u_0$ weakly in $L^{p^*}(\Omega)$.

Using the affine Sobolev inequality on $W^{1,p}_0(\Omega)$,
\[
K_{n,p}^{-p} \left( \int_{\Omega} |u|^{p^*}\, dx\right)^{\frac{p}{p^*}} \leq {\cal E}^p_{p, \Omega}(u),
\]
we get
\[
c_{\cal A} = \lim_{k \rightarrow \infty}\left( {\cal E}^p_{p, \Omega}(u_k) - \lambda \int_{\Omega} |u_k|^p\, dx \right) \geq  K_{n,p}^{-p} - \lambda \int_{\Omega} |u_0|^p\, dx,
\]
so the condition $c_{\cal A} < K_{n,p}^{-p}$ implies that $u_0 \neq 0$. Hence, by Corollary \ref{C2} and Theorem \ref{T6}, we have $u_k \rightharpoonup u_0$ weakly in $W^{1,p}_0(\Omega)$ and $\Phi_{\cal A}(u_0) \leq c_{\cal A}$. It only remains to show that $u_0 \in X$.

Evoking Proposition \ref{P5}, we easily deduce that

\begin{eqnarray*}
c_{\cal A} &=& \lim_{k \rightarrow \infty} \Phi_{\cal A}(u_k) \\
&\geq & \lim_{k \rightarrow \infty} \left( \Phi_{\cal A}(T_h u_k) + \Phi_{\cal A}(R_h u_k) \right)\\
&\geq & c_{\cal A}\lim_{k \rightarrow \infty}\left( \Vert T_h u_k\Vert^p_{L^{p^*}(\Omega)} + \Vert R_h u_k\Vert^p_{L^{p^*}(\Omega)} \right).
\end{eqnarray*}
Applying now Lemma 3.1 of \cite{BW}, we derive

\begin{eqnarray*}
c_{\cal A} &\geq & c_{\cal A}\left[ \Vert T_h u_0 \Vert^p_{L^{p^*}(\Omega)} + \left( 1 + \Vert R_h u_0 \Vert_{L^{p^*}(\Omega)}^{p^*} - \Vert u_0 \Vert_{L^{p^*}(\Omega)}^{p^*} \right)^{\frac{p}{p^*}} \right].
\end{eqnarray*}
Using the condition $c_{\cal A} > 0$ and letting $h \rightarrow \infty$, we obtain

\[
1\geq \left( \Vert u_0 \Vert_{L^{p^*}(\Omega)}^{p^*} \right)^{\frac{p}{p^*}} + \left( 1 - \Vert u_0 \Vert_{L^{p^*}(\Omega)}^{p^*}\right)^{\frac{p}{p^*}},
\]
and thus $u_0 \in X$ because $u_0 \neq 0$.
\end{proof}

\begin{propo}\label{P7}
The equality ${\cal E}_p(u) = \Vert \nabla u \Vert_{L^p(\R^n)}$ is valid for every radial function $u \in {\cal D}^{1,p}(\R^n)$.
\end{propo}

\begin{proof}
It suffices to prove the result for $u \neq 0$. By (iii) of Proposition \ref{P1}, we recall that $0 < c \leq \Psi_\xi(u) \leq C$ for all $\xi \in \s^{n-1}$, where $c$ and $C$ are positive constants.

We remark that the proof of the inequality ${\cal E}_p(u) \leq \Vert \nabla u \Vert_{L^p(\R^n)}$ involves two independent steps. The first one consists in applying H\"{o}lder's inequality as follows:

\begin{eqnarray*}
n \omega_n = \int_{\s^{n-1}} 1\, d\sigma(\xi) &=& \int_{\s^{n-1}} \Psi_\xi(u)^{\frac{n}{n+p}} \Psi_\xi(u)^{-\frac{n}{n+p}}\, d\sigma(\xi) \\
&\leq& \left( \int_{\s^{n-1}} \Psi_\xi(u)\, d\sigma(\xi) \right)^{\frac{n}{n+p}} \left( \int_{\s^{n-1}} \Psi_\xi(u)^{-\frac{n}{p}}\, d\sigma(\xi) \right)^{\frac{p}{n+p}},
\end{eqnarray*}
which yields
\begin{equation} \label{hol}
(n \omega_n)^{\frac{n+p}{n}} \left( \int_{\s^{n-1}} \Psi_\xi(u)^{-\frac{n}{p}}\, d\sigma(\xi) \right)^{-\frac{p}{n}} \leq \int_{\s^{n-1}} \Psi_\xi(u)\, d\sigma(\xi).
\end{equation}
The second step makes use of the Fubini's theorem on the above right-hand side, so we get
\begin{eqnarray}
\int_{\s^{n-1}} \Psi_\xi(u) \, d\sigma(\xi) &=& \int_{\s^{n-1}} \int_{\R^n} | \nabla u(x) \cdot \xi |^p\, dx\, d\sigma(\xi) \nonumber \\
&=& \int_{\R^n} \int_{\s^{n-1}} | \nabla u(x) \cdot \xi |^p\, d\sigma(\xi)\, dx \nonumber \\
&=& \left( \int_{\s^{n-1}} | \xi_0 \cdot \xi |^p\, d\sigma(\xi) \right) \left( \int_{\R^n} | \nabla u(x) |^p\, dx \right), \label{fub}
\end{eqnarray}
where $\xi_0$ is any fixed point $\xi_0$ in $\s^{n-1}$.

On the other hand, we know from \cite{LYZ2} that
\[
\alpha_{n,p} = (n \omega_n)^{\frac{n+p}{np}} \left( \int_{\s^{n-1}} | \xi_0 \cdot \xi |^p\, d\sigma(\xi) \right)^{- \frac 1p}.
\]
Hence, joining \eqref{hol} and \eqref{fub}, we obtain ${\cal E}_p(u) \leq \Vert \nabla u \Vert_{L^p(\R^n)}$. Moreover, equality holds if, and only if, the step of the H\"{o}lder's inequality becomes equality. But the latter is equivalent to the function $\xi \in \s^{n-1} \mapsto \Psi_\xi(u)$ to be constant.

Finally, for any radial function $u \in {\cal D}^{1,p}(\R^n)$, it easily follows that

\[
\Psi_\xi(u) = \int_{\R^n} \vert \nabla u(x) \cdot \xi \vert^p\, dx = \int_{\R^n} r^{-p} \vert u^\prime(r) \vert^p \vert x \cdot \xi \vert^p\, dx = \int_{\R^n} r^{-p} \vert u^\prime(r) \vert^p \vert x \cdot \xi_0 \vert^p\, dx
\]
for all $\xi \in \s^{n-1}$. In other words, $\Psi_\xi(u)$ doest'n depend on $\xi$ and this concludes the proof.
\end{proof}

\begin{proof}[Proof of Theorems \ref{T2} and \ref{T3}]
Arguing as in the proof of Theorem \ref{T1} with the aid of Propositions \ref{P2} and \ref{P3}, it suffices to establish the existence of a minimizer of $\Phi_{\cal A}$ in $X$.

We first assert that the condition $\lambda < \lambda^{\cal A}_{1,p}$ implies that $c_{\cal A} > 0$ in the critical case too. Indeed, since $q = p^*$, the sharp affine $L^p$ Poincaré and Sobolev inequalities lead to

\[
\Phi_{\cal A}(u) = {\cal E}^p_{p, \Omega}(u) - \lambda \int_{\Omega} |u|^p\, dx \geq \left\{ \begin{array}{llll}
{\cal E}^p_{p, \Omega}(u) \geq K_{n,p}^{-p}, &{\rm if}\ \lambda \leq 0, \\
( 1 - \frac{\lambda}{\lambda^{\cal A}_{1,p}} ) {\cal E}^p_{p, \Omega}(u) \geq ( 1 - \frac{\lambda}{\lambda^{\cal A}_{1,p}} ) K_{n,p}^{-p}, &{\rm if}\ \lambda > 0
\end{array}\right.
\]
for every $u \in X$. Therefore, $c_{\cal A} > 0$ in any situation.

We now show that $c_{\cal A} < K_{n,p}^{-p}$ provided that $\lambda > 0$. By Proposition \ref{P7}, the best constant $K_{n,p}$ for the affine inequality \eqref{ASob} coincides with the corresponding one for the classical $L^p$ Sobolev inequality.

For $n \geq p^2$, we evoke the well-known construction by Azorero and Peral \cite{AP1} of a function $w_0 \in X$ such that
\[
c_{\cal A} \leq \Phi_{\cal A}(w_0) = {\cal E}^p_{p, \Omega}(w_0) - \lambda \int_{\Omega} |w_0|^p\, dx \leq \int_\Omega \vert \nabla w_0 \vert^p\, dx - \lambda \int_{\Omega} |w_0|^p\, dx < K_{n,p}^{-p}.
\]
Here it was used \eqref{comp}.

For $n < p^2$, we take a principal eigenfunction $\varphi^{\cal A}_{1,p}$ of $\Delta^{\cal A}_p$ on $W^{1,p}_0(\Omega)$ with $\Vert \varphi^{\cal A}_{1,p} \Vert_{L^{p^*}(\Omega)} = 1$, whose existence is ensured by Theorem 4 of \cite{HJM5}, so $\varphi^{\cal A}_{1,p} \in X$. Set $\lambda_* = \lambda^{\cal A}_{1,p} - K_{n,p}^{-p} \Vert \varphi^{\cal A}_{1,p} \Vert^{-p}_{L^p(\Omega)}$. Note that $\lambda_* > 0$, because
\[
\lambda^{\cal A}_{1,p} \Vert \varphi^{\cal A}_{1,p} \Vert^p_{L^p(\Omega)} = {\cal E}^p_{p, \Omega}(\varphi^{\cal A}_{1,p}) > K_{n,p}^{-p},
\]
where the strict inequality follows from the characterization of extremals as quoted in the introduction. For $\lambda > \lambda_*$ and $w_0 = \varphi^{\cal A}_{1,p}$, we obtain

\begin{eqnarray*}
c_{\cal A} \leq \Phi_{\cal A}(w_0) &=& {\cal E}^p_{p, \Omega}(w_0) - \lambda \int_{\Omega} |w_0|^p\, dx \\
&\leq& \lambda^{\cal A}_{1,p} \Vert \varphi^{\cal A}_{1,p} \Vert^p_{L^p(\Omega)} - \lambda \int_{\Omega} |w_0|^p\, dx\\
&<& \lambda^{\cal A}_{1,p} \Vert \varphi^{\cal A}_{1,p} \Vert^p_{L^p(\Omega)} - \lambda_* \Vert \varphi^{\cal A}_{1,p} \Vert^p_{L^p(\Omega)} = K_{n,p}^{-p}.
\end{eqnarray*}

Finally, under the assumptions of Theorems \ref{T2} and \ref{T3}, we deduce that $0 < c_{\cal A} < K_{n,p}^{-p}$ and hence, by Proposition \ref{P6}, we complete the proof.
\end{proof}

\section{Proof of nonexistence theorems}

We prove Theorems \ref{T4} and \ref{T5} by using Propositions \ref{P2}, \ref{P3} and \ref{P4}.

\begin{proof}[Proof of Theorem \ref{T4}]
Let $p < q \leq p^*$ and $\lambda \geq \lambda^{\cal A}_{1,p}$. Assume that the problem \eqref{1.1} admits a nontrivial least energy weak solution $u_0 \in W^{1,p}_0(\Omega)$. Then,

\[
c_{\cal A} = \frac{{\cal E}^p_{p, \Omega}(u_0) - \lambda \int_{\Omega} |u_0|^p\, dx}{\left( \int_{\Omega} |u_0|^q\, dx \right)^{\frac pq}} = \frac{\int_{\Omega} |u_0|^q\, dx}{\left( \int_{\Omega} |u_0|^q\, dx \right)^{\frac pq}} > 0.
\]
On the other hand, for a principal eigenfunction $\varphi^{\cal A}_{1,p}$ of $\Delta^{\cal A}_p$ on $W^{1,p}_0(\Omega)$ with $\Vert \varphi^{\cal A}_{1,p} \Vert_{L^{q}(\Omega)} = 1$, we get

\[
c_{\cal A} \leq {\cal E}^p_{p, \Omega}(\varphi^{\cal A}_{1,p}) - \lambda \int_{\Omega} (\varphi^{\cal A}_{1,p})^p\, dx = (\lambda^{\cal A}_{1,p} - \lambda) \int_{\Omega} (\varphi^{\cal A}_{1,p})^p\, dx \leq 0,
\]
and thus we derive a contradiction.
\end{proof}

\begin{proof}[Proof of Theorem \ref{T5}]
Let $u_0 \in W^{1,p}_0(\Omega)$ be a nontrivial nonnegative weak solution of \eqref{1.1}. By Propositions \ref{P2} and \ref{P3}, $u_0$ is a positive $C^1$ function that satisfies $\nabla u_0 \neq 0$ on $\partial \Omega$. Assume without loss of generality that $\Omega$ is star-shaped with respect to the origin. Thus, we have $x \cdot \nu > 0$ on $\partial \Omega$. Since $q = p^*$, applying Proposition \ref{P4}, we get the contradiction
\begin{eqnarray*}
0 > \left( \frac{1}{p} - 1 \right) \int_{\partial \Omega} H_{u_0}^p(\nabla u_0) (x \cdot \nu)\, d\sigma &=& \left( \frac{n}{p} - 1 \right) \int_\Omega u_0^q + \lambda u_0^p\, dx - n \int_\Omega \frac{1}{q} u_0^q + \frac{\lambda}{p} u_0^p\, dx\\
&=& \left( \frac{n - p}{p} - \frac{n}{q} \right) \int_\Omega u_0^q\, dx - \lambda \int_\Omega u_0^p\, dx\\
&=& - \lambda \int_\Omega u_0^p\, dx \geq 0
\end{eqnarray*}
for every $\lambda \leq 0$.
\end{proof}

\n {\bf Acknowledgments:} The first author was partially supported by CNPq/Brazil (PQ 316526/2021-5) and  Fapemig (Universal-APQ-00709-18) and the second author was partially supported by CNPq (PQ 302670/2019-0 and Universal 429870/2018-3) and Fapemig (PPM 00561-18).

\end{document}